\documentclass[12pt]{article}
\usepackage{e-jcb}

\usepackage{amsmath}
\usepackage{amssymb}
\usepackage{amsthm}
\usepackage{graphicx}

\usepackage[colorlinks=true,citecolor=black,linkcolor=black,urlcolor=blue]{hyperref}

\DeclareMathOperator{\lcm}{lcm}

\theoremstyle{plain}
\newtheorem{theorem}{Theorem}

\newtheorem{cor}[theorem]{Corollary}
\newtheorem{lemma}[theorem]{Lemma}

\newtheorem{conj}[theorem]{Conjecture}

\theoremstyle{definition}
\newtheorem{definition}[theorem]{Definition}

\newtheorem{remark}[theorem]{Remark}

\numberwithin{theorem}{section}

\providecommand{\floor}[1]{\left\lfloor#1\right\rfloor}
\providecommand{\Z}{\mathbb{Z}} \providecommand{\R}{\mathbb{R}}
 
\providecommand{\Q}{{\mathbb{Q}}}

\newcommand{\setBuilder}[2]{\left\{\vphantom{x^2}{#1}\,\middle|\,{#2}\right\}}

\newcommand{\ideal}[1]{\left\langle{#1}\right\rangle}

\newcommand{\proofPart}[1]{{\bf $\boldsymbol{#1}$:}}

\newcommand{\proofCase}[1]{\proofPart{\text{The case }#1}}

\DeclareMathOperator{\eqp}{EQP}

\newcommand{\defeq}{=}

\title{The Parametric Frobenius Problem\thanks{Published in the Electronic Journal of Combinatorics 22 (2015), \#P2.36.}}
\author{Bjarke Hammersholt Roune\thanks{Currently at Google. Research supported by Algorithmische und Experimentelle Methoden in Algebra, Geometrie und Zahlentheorie (SPP 1489)}\\
\small Department of Mathematics\\[-0.8ex]
\small University of Kaiserslautern\\[-0.8ex] 
\small Kaiserslautern, Germany\\
\small\tt bjarke.roune@gmail.com\\
\and
Kevin Woods\\
\small Department of Mathematics\\[-0.8ex]
\small Oberlin College\\[-0.8ex]
\small Oberlin, Ohio, USA\\
\small\tt Kevin.Woods@oberlin.edu
}

\date{\small Mathematics Subject Classifications: 11D07, 52C07, 11H06}

\begin{document}

\maketitle 

\begin{abstract}
Given relatively prime positive integers $a_1,\ldots,a_n$, the Frobenius number is the largest integer that cannot be written as a nonnegative integer combination of the $a_i$. We examine the parametric version of this problem: given $a_i=a_i(t)$ as functions of $t$, compute the Frobenius number as a function of $t$. A function $f:\Z_+\rightarrow\Z$ is a quasi-polynomial if there exists a period $m$ and polynomials $f_0,\ldots,f_{m-1}$ such that $f(t)=f_{t\bmod m}(t)$ for all $t$. We conjecture that, if the $a_i(t)$ are polynomials (or quasi-polynomials) in $t$, then the Frobenius number agrees with a quasi-polynomial, for sufficiently large $t$. We prove this in the case where the $a_i(t)$ are linear functions, and also prove it in the case where $n$ (the number of generators) is at most 3.
\end{abstract}

\section{Introduction}
Given positive integers $a_i$, $1\le i\le n$, let
\[\ideal{a_1,\ldots, a_n}\defeq\setBuilder{\sum_{i=1}^n p_ia_i}{p_i\in\Z_{\ge 0}}\]
be the semigroup generated by the
$a_i$. If the $a_i$ are relatively prime, define the \emph{Frobenius
number} $F(a_1,\ldots,a_n)$ to be the largest integer not in
$\ideal{a_1,\ldots, a_n}$. The \emph{Frobenius problem} of determining $F(a_1,\ldots,a_n)$ has a long history --- Sylvester proved \cite{Sylvester84} in 1884 that $F(a,b)=ab-a-b$, and see Ram{\'{\i}}rez Alfons{\'{\i}}n's book \cite{RA05} for many subsequent results.

It will be convenient to also define $F(a_1,\ldots,a_n)$ in the case where the $a_i$ are not relatively prime, so that there is no largest integer not in the semigroup. A reasonable definition seems to be the largest integer in the group $\Z a_1+\cdots+\Z a_n$ that is not in the semigroup $\ideal{a_1,\ldots, a_n}$. That is, if $d$ is the greatest common divisor of $a_1,\ldots,a_n$, then $F(a_1,\ldots,a_n)=dF(\frac{a_1}{d},\ldots,\frac{a_n}{d})$. Sylvester's identity then becomes $F(a,b)=\lcm(a,b)-a-b$.

The \emph{parametric Frobenius problem} is, given functions $a_i:\Z_+\rightarrow\Z_+$ to determine $F\big(a_1(t),\ldots,a_n(t)\big)$ as a function of $t$. For example, using Sylvester's identity,
\[F(t,t+2)=\begin{cases}t(t+2)-t-(t+2) & \text{if $t$ is odd,}\\\frac{t(t+2)}{2} - t -(t+2) & \text{if $t$ is even.}\end{cases}\]

We see that $F(t,t+2)$ is a \emph{quasi-polynomial}; a function $f:\Z_+\rightarrow\Z$ is a quasi-polynomial if
there exist an $m\in\Z_+$ and polynomials $f_0,\ldots,f_{m-1}\in\Q[t]$
such that $f(t)=f_{t\bmod m}(t)$ for all $t\in\Z_+$. Here $m$ is a
\emph{period} of $f$ and the $f_i$ are \emph{components} of $f$. We will assume that all of our functions are integer-valued, but note that the $F(t,t+2)$ example shows that we may need polynomials with rational coefficients.

In general, our functions may misbehave for small $t$. We say that a property is \emph{eventually} true if it is true for all sufficiently large $t$. We say that  $f(t)\in\eqp$ (short for \emph{eventual quasi-polynomial}) if $f(t)$ eventually agrees with a quasi-polynomial. We say that the \emph{degree} of $f\in\eqp$ is the maximum degree of its components.
 
 We conjecture that, if $a_i(t)\in\eqp$, then $F\big(a_1(t),\ldots,a_n(t)\big)\in\eqp$. We prove this for the special case where the $a_i(t)$ are linear functions and for the special case where $n\le 3$.

\begin{conj}
Suppose that $a_i(t)\in\eqp$, $1\le i\le n$, are eventually positive. Then $F\big(a_1(t),\ldots,a_n(t)\big)\in\eqp$.
\end{conj}

\begin{theorem}
\label{thm:main}
Suppose that $a_i(t)\in\eqp$, $1\le i\le n$, are eventually positive and have degree at most 1. Then $F\big(a_1(t),\ldots,a_n(t)\big)\in\eqp$, with degree at most 2.
\end{theorem}

\begin{theorem}
\label{thm:3gens}
Suppose that $a_i(t)\in\eqp$, $1\le i\le n$, are eventually positive, with $n\le 3$. Then $F\big(a_1(t),\ldots,a_n(t)\big)\in\eqp$.

\end{theorem}

These results are examples of so-called ``unreasonable'' appearances of qua\-si-po\-ly\-no\-mi\-als, as discussed by Woods \cite{Woods14}. ``Reasonable'' appearances trace back to Ehrhart's classical result \cite{Ehrhart62} that, if $P\subseteq\R^n$ is a polytope with rational vertices, then $f(t)=\#(tP\cap\Z^n)$ is a quasi-polynomial. Note that if $P$ is defined with linear inequalities $\mathbf b_i\cdot \mathbf x\le c_i$, then $tP$ is defined with linear inequalities $\mathbf b_i\cdot \mathbf x\le c_it$. As $t$ changes, these linear inequalities move, but their normal vectors ($\mathbf b_i$) remain the same. Indeed, Woods proved \cite{Woods15} that any example over the integers defined with linear inequalities, boolean operations (and, or, not), and quantifiers ($\forall$, $\exists$) has this quasi-polynomial behavior. This is true even if there is more than one parameter; for example,
\[\#\setBuilder{(x,y)\in\Z^2_{\ge 0}}{2x\le t,\, 3y\le s} = \bigg(\floor{\frac{t}{2}+1}\bigg)\bigg(\floor{\frac{s}{3}+1}\bigg)\]
is a quasi-polynomial of period 2 in $t$ and period 3 in $s$.

If we look at the parametric Frobenius problem, however, we see that it does not fit this pattern. For example, if we want to ask whether $u\in\ideal{t,t+1,t+2}$, we are asking whether the polytope
\[\setBuilder{(x,y,z)\in\R^3_{\ge 0}}{u=tx + (t+1)y+(t+2)z}\]
contains any integer points. For a fixed $t$, this is a 2-dimensional triangle in $\R^3$; as $t$ changes, this triangle ``twists'' (the normal vector changes). Examples such as this were categorized in \cite{Woods14} as ``unreasonable'', though they are conjectured to still lead to eventual quasi-polynomial behavior.

This paper adds a third example to the list of recent  results demonstrating this phenomenon; previously Chen, Li, and Sam \cite{CLS12} showed that the number of integer points in a polytope whose vertices are rational functions of $t$ is in $\eqp$,  and Calegari and Walker \cite{CW11} showed that the vertices of the integer hull of such a polytope have coordinates in $\eqp$. A critical tool used in all of these results is that the division algorithm and the gcd of polynomials has quasi-polynomial behavior (cf. Lemma \ref{lemma:tools}); for example, the Euclidean algorithm yields
\[\gcd(2t+1, 5t+6)=\gcd(t+4, 2t+1)=\gcd(7,t+4)=\begin{cases}7& \text{if $t\equiv 3\bmod 7$,}\\1&\text{otherwise.}\end{cases}\]

Note that unlike in the ``reasonable'' case, these results only hold with one parameter variable. For example $F(s,t)=\lcm(s,t)-s-t$ is not a quasi-polynomial in $s$ and $t$.

The original inspiration for this paper comes from a conjecture of Wagon \cite{W15} (see \cite[Section 17]{ELSW07}) that, for any fixed $M$ and residue class $j$ of $t\bmod M^2$, there exist (usually positive) integers $c_{M,j}$ and $d_{M,j}$ such that eventually
\[F(t,t+1^2,t+2^2,\ldots,t+M^2)=\frac{1}{M^2}\left(t^2+c_{M,j}t\right)-d_{M,j}.\]
Using the proof of Theorem \ref{thm:main}, we can prove that there is indeed quasi-polynomial behavior with period $M^2$. In fact, such a result holds in greater generality:

\begin{cor}
\label{cor:wagon}
Suppose $b_1<b_2<b_3<\cdots<b_n$ are integers. Then $F(t+b_1, t+b_2,\ldots,t+b_n)\in\eqp$ with period $b_n-b_1$.
\end{cor}

When the $b_i$ form an arithmetic sequence, a precise quasi-polynomial formula was previously given by Roberts \cite{Roberts56}:
\[F(t,t+d,\ldots,t+sd)=\left(\floor{\frac{t-2}{s}}+1\right)t+(d-1)(t-1)-1.\]

In Section 2, we work through the example
\[F(t,t+1,t+2)=\left(\floor{\frac{t-2}{2}}+1\right)t-1=\begin{cases} \frac{t^2}{2} & \text{if $t$ even,}\\ \frac{t^2}{2}-\frac{t}{2}-1 & \text{if $t$ odd,}\end{cases}\]
which will give a flavor of our general proof. In Section 3, we prove Theorem \ref{thm:main}. In Section 4, we prove the various lemmas needed. In Section 5, we prove Corollary \ref{cor:wagon}. In Section 6, we prove Theorem \ref{thm:3gens}, using R\o dseth's algorithm \cite{Rodseth78} for the 3 generator Frobenius problem.

The proof of Theorem \ref{thm:main} relies on the fact that semigroups with only two generators are much easier to deal with. Indeed, we have the following definition and lemma:

\begin{definition}
\label{def:canon}
Let $a,b\in\Z_+$ be relatively prime, and let $c\in\Z$. The \emph{canonical form} for $c$ is given by $c=pa+qb$ with $p,q\in \Z$ and $0\le p<b$.
\end{definition}

\begin{lemma}
\label{lemma:canon}
Let $a,b\in\Z_+$ be relatively prime, and let $c\in\Z$.

\begin{enumerate}
\item The canonical form for $c$ exists and is unique. In particular, if $c=p'a+q'b$ is any form with $p',q'\in \Z$,  and if $k$ and $r$ are the quotient and remainder when $p'$ is divided by $b$, then the canonical form for $c$ is $ra+(q'+ka)b$.
\item If $c=pa+qb$  is in canonical form, $c\in\ideal{a,b}$ if and only if $q\ge 0$.
\end{enumerate}
\end{lemma}

\section{An Example}
In this example, we compute $F(t,t+1,t+2)$. Let $a=t$, $b=t+1$, and $c=t+2$ be our three generators, let $S\defeq\ideal{a,b,c}$, and let $T\defeq\ideal{a,c}$. Notice that $2b=a+c$. This implies that $S=T\cup (b+T)$, as follows: if $pa+qb+rc$ is a representation of an element of $S$, with $p,q,r\in\Z_{\ge 0}$ and $q\ge 2$, then $(p+1)a+(q-2)b+(r+1)c$ is also a representation, so we may assume without loss of generality that $q$ is $0$ or $1$ (cf. Lemma \ref{lemma:decomp}). Next, we run the extended Euclidean algorithm on the integers $a$ and $c$, and get
\[\gcd(a,c)=\gcd(t,t+2)=\gcd(2,t)\]
(cf. Lemma \ref{lemma:tools}). The next division step in the algorithm depends on the parity of $t$, and rather than end up with the messy $\floor{t/2}$, we divide into two cases.

\bigskip

\proofCase{t \text{ is odd}} Let $t=2s+1$, so that $a=2s+1$, $b=2s+2$, and $c=2s+3$. Now
\[\gcd(a,c)=\gcd(2,2s+1)=\gcd(1,2)=1,\]
and the extended Euclidean algorithm yields $1=(s+1)a-sc$. Let $u\in \Z$, and we are wondering whether $u\notin S$, that is, $u\notin T$ and $u\notin b+T$. Suppose that
\[u=pa+qc\]
 is the canonical form for $u$ (see Definition \ref{def:canon}). Lemma \ref{lemma:canon}(2) tells us that $u \in T$ if and only if $q\ge 0$. Since we are looking for $u\notin T$, we may assume from here on out that $q<0$. To characterize when $u\in b+T$ (that is, $u-b\in T$) we must find the canonical form for $u-b$.

First we compute the canonical form for $b$. Multiplying the equation $1=(s+1)a-sc$ by $b=2s+2$ yields \emph{some} form for $b$:
\[b=(2s+2)(s+1)a-(2s+2)sc.\]
By Lemma \ref{lemma:canon}(1), we may find the canonical form by dividing $(2s+2)(s+1)$ by $c=2s+3$: the quotient is $s$ with remainder $s+2$, giving the canonical form for $b$ as
\[b=(s+2)a-sc.\] Note that we got a little lucky here: our remainder of $s+2$ is clearly less than $c=2s+3$; if our remainder had instead been $s+7$, say, then we would only have the canonical form for sufficiently large $s$.

Now we have that \emph{some} form for $u-b$ is
\[u-b=(p-s-2)a+(q+s)c.\]
Is this the canonical form? There are two cases:

If $p\ge s+2$, then $0\le p-s-2\le p<c$, and this is in canonical form. Therefore, to have $u-b\in T$, we must have $q+s\ge 0$ (again using Lemma \ref{lemma:canon}(2)), that is, $q\ge -s$.

 If $p<s+2$, then the canonical form is
\[u-b=(p-s-2+c)a+(q+s-a)c=(p+s+1)a+(q-s-1)c\]
(to check that this is canonical, note that $p\ge 0$ and $s+2<c$ imply $p-s-2+c\ge 0$, and $p<s+2$ implies $p-s-2+c<c$). In this case, to have $u-b\in T$, we must have $q\ge s+1$. This implies that $q\ge 0$, but we have assumed $q< 0$ (so that $u\notin T$). Therefore this case never has $u-b\in T$.

\begin{figure}
\begin{center}
\includegraphics[width=2in]{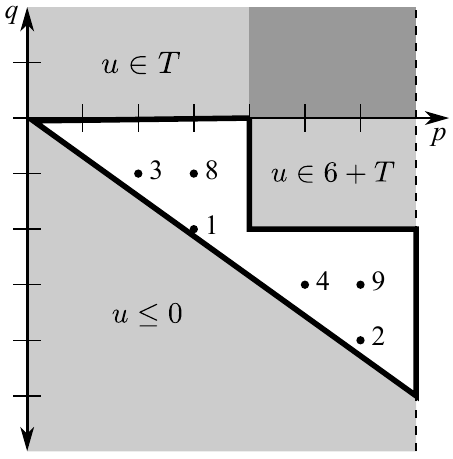}
\caption{The $t=5$ case of the example. We have $T=\ideal{5,7}$ and $S=\ideal{5,6,7}=T\cup(6+T)$. Points $(p,q)\in\Z^2$ with $0\le p<7$ are the canonical forms for $u=5p+7q\in \Z$. Then $q\ge 0$ corresponds to $u\in T$, $(p,q)\ge (4,-2)$ corresponds to $u\in 6+T$, and positive $u$ such that $u\notin S$ are labelled beside their corresponding $(p,q)$. The ``corners'' $u=8$ and $u=9$ are candidates for the Frobenius number, and so $F(5,6,7)=\max\{8,9\}=9$.}
\label{fig:ex}
\end{center}
\end{figure}

To summarize across both cases, if $u=pa+qc$, with $0\le p < c=2s+3$ in the canonical form, then $u\notin S$ if and only if
\[ \text{not}(q\ge 0)\quad\text{and}\quad\text{not}\big((p,q)\ge(s+2,-s)\big)\]
(cf. Lemma \ref{lemma:oneh}).
The set of such $(p,q)$ has a ``stairstep'' shape, as can be seen in Figure \ref{fig:ex} for $t=5$. In particular, the set of such $(p,q)$ can be rewritten as $p\ge 0$ and 
\[\big(p\le s+1\quad\text{and}\quad q\le -1\big)\quad\text{or}\quad\big(p\le 2s+2\quad\text{and}\quad q\le-s-1\big)\]
(cf. proof of Lemma \ref{lemma:finish}, $d=1$ case). The two ``corners'' in this picture, where both inequalities of one of these two  conjunctions are tight, give our candidates for the largest $u\notin S$: it must be either
\[(s+1)a-1c=2s^2+s-2 \quad\text{or}\quad (2s+2)a+(-s-1)c=2s^2+s-1.\]
The latter is always larger (in general, one might only be eventually larger than the other), and so we have
\[F(t,t+1,t+2)=2s^2+s-1=\frac{t^2}{2}-\frac{t}{2}-1,\]
in this case where $t=2s+1$ is odd.

\bigskip

\proofCase{t\text{ is even}} Let $t=2s$, so that $a=2s$, $b=2s+1$, and $c=2s+2$.
To have $u\notin S$ we must have both $u\notin T=\ideal{a,c}$ and $u\notin b+T$. Note that every element in $T$ is even and every element of $b+T$ is odd. Therefore the largest even integer not in S is the largest even integer not in $T$, which since $T$ has only 2 generators is simply $\lcm(a,c)-a-c=2s^2-2s-2$ (we're getting lucky here that we didn't have to do an analysis similar to the $t$ is odd case). Similarly the largest odd integer not in $S$ is the largest odd integer not in $b+T$, which is $b+(2s^2-2s-2)=2s^2-1$. The largest integer not in $S$ is then the maximum of these two candidates, so $F(t,t+1,t+2)=2s^2-1=\frac{t^2}{2}-1$ (cf. proof of Lemma \ref{lemma:finish}, $d>1$ case), in this case where $t$ is even.

\bigskip

Combining the even and odd case, we have proved that
\[F(t,t+1,t+2)=\begin{cases} \frac{t^2}{2} & \text{if $t$ even,}\\ \frac{t^2}{2}-\frac{t}{2}-1 & \text{if $t$ odd.}\end{cases}\]

\section{Proof of Theorem \ref{thm:main}}
This first lemma gives us the basic tools we need: 
\begin{lemma}
\label{lemma:tools} Given $f,g\in \eqp$, 
\begin{enumerate}
\item There exists an $m\in\Z_+$ such that, for each $0\le j< m$, $f(ms+j)$ eventually agrees with a polynomial in $\Z[s]$.
\item If $\deg(g)>0$, there exists $q,r\in\eqp$ such that $f(t)=q(t)g(t)+r(t)$ and $\deg(r)<\deg(g)$. Note that this is the analogue of the traditional division algorithm over $\Q[t]$.
\item If $g(t)$ is eventually positive, there exists $q,r\in\eqp$ such that $f(t)=q(t)g(t)+r(t)$ and eventually $0\le r(t)<g(t)$. Furthermore, $\deg(r)\le \deg(g)$. Note that this is the analogue of the traditional division algorithm over $\Z$.
\item There exists $p,q,d\in\eqp$ such that $\gcd\big(f(t),g(t)\big)=d(t)$ and $d(t)=p(t)f(t)+q(t)g(t)$.
\item We have $\max\big(f(t),g(t)\big)\in\eqp$.
\end{enumerate}
\end{lemma}

These are proved in Section 4 of \cite{CLS12} (and also in \cite{CW11}), so in Section 4 we merely give an outline of the important steps.

\begin{remark}
\label{remark:wlog}
Lemma \ref{lemma:tools}(1) will allow us to often simply say ``without loss of generality, $f\in\Z[t]$'': we may analyze $f(ms+j)$ for each $j$, recognize that statements may be false for small $s$, and then convert back to $t$ using $s=(t-j)/m$.
\end{remark}

Using this remark, we may assume that the $a_i(t)$ are polynomials in $\Z[t]$. Since they are of degree at most 1 and eventually positive, we have that 
\[a_i(t)=\alpha_i t + \beta_i,\]
 with either $\alpha_i\in\Z_{+}$,
$\beta_i\in\Z$ or $\alpha_i=0$, $\beta_i\in\Z_+$.
Furthermore, without loss of generality, we may assume that they are \emph{ordered} so that
\[\beta_i\alpha_j\le \beta_j\alpha_i\]
for $i\le j$, that is (for $\alpha_i\ne 0$),
\[
\frac{\beta_1}{\alpha_1} \le
\frac{\beta_2}{\alpha_2} \le
\cdots \le \frac{\beta_n}{\alpha_n}.
\]
We first consider a degenerate case, where
$\beta_i\alpha_j=\beta_j\alpha_i$, for all $i,j$. 

\begin{lemma}
\label{lemma:degen}
If $\beta_i\alpha_j=\beta_j\alpha_i$, for all $i,j$, then there are $\alpha_0, \beta_0\in\Z_{\ge 0}$ such that, for all $i$,  $\alpha_i=\gamma_i\alpha_0$ and $\beta_i=\gamma_i\beta_0$ (for some $\gamma_i\in\Z_+$). Therefore 
\[F\big(a_1(t),\ldots,a_n(t)\big)=F(\gamma_1,\ldots,\gamma_n)\cdot(\alpha_0t+\beta_0)\]
is a polynomial of degree at most 1.
\end{lemma}

Therefore, we can assume that $\beta_1\alpha_n<\beta_n\alpha_1$. In particular, the polynomials $a_1$ and $a_n$ do not share a linear factor, and so if $d(t)=\gcd\big(a_1(t),a_n(t)\big)$, then $d(t)\in\eqp$ is of degree 0, that is, $d(t)$ is eventually a periodic function. Let
$S(t)\defeq\ideal{a_1(t),\ldots,a_n(t)}$, and let $T(t)\defeq\ideal{
a_1(t),a_n(t)}$. $T(t)$ is a semigroup with only two generators, so
it is much easier to analyze. The following lemma allows us to cover $S(t)$
with a finite number of translated copies of $T(t)$.

\begin{lemma}
\label{lemma:decomp}
For $a_i(t)=\alpha_it+\beta_i$ as described above,  there exists a finite set $H$
of integer-valued polynomials of degree at most 1 such that
\[
\ideal{a_1(t),\ldots,a_n(t)}=\bigcup_{h\in H}
\Big(h(t) + \ideal{a_1(t),a_n(t)} \Big)
\]
(for $t$ sufficiently large so that $a_i(t)>0, \forall i$).
Furthermore, $0\in H$ and the other $h\in H$ are eventually positive.
\end{lemma}

Assume for the moment that $a_1(t)$ and $a_n(t)$ are relatively prime. By Lemma \ref{lemma:canon}(2), the set of all integers not in $T$ will be in bijection to the set
\[\setBuilder{(p,q)\in\Z^2}{0\le p< a_n(t),\, q<0},\]
under the bijection $(p,q)\mapsto pa_1(t)+qa_n(t)$. By Lemma \ref{lemma:decomp}, integers not in $S$ must also not be in $h(t)+T$, for all $h\in H$. The following lemma gives an easy way to check this.

\begin{lemma}
\label{lemma:oneh}
Let $f,g,h\in\eqp$, with $f(t)$ and $g(t)$ relatively prime for all $t$ and with $f(t),g(t),h(t)$ eventually positive. There exists $r,s\in\eqp$ such that, given $p,q\in\Z$ with $0\le p< g(t)$ and $q<0$,
\[pf(t)+qg(t)\in h(t)+\ideal{f(t),g(t)}\quad\text{if and only if}\quad (p,q)\ge \big(r(t),s(t)\big)\]
component-wise. Furthermore, $\deg(r)\le\deg(g)$ and $\deg(s)\le\max\{\deg(f),\deg(h)-\deg(g)\}$,
\end{lemma}

In our case, this gives $r,s$ of degree at most 1. If $a_1(t)$ and $a_n(t)$ are relatively prime, all that remains is to follow the implications of Lemmas \ref{lemma:decomp} and Lemma \ref{lemma:oneh}, which will give us a staircase-shaped set such as in Figure \ref{fig:ex}. If $a_1(t)$ and $a_n(t)$ are not relatively prime, we must reduce to the case where they are. Both of these are accomplished in the following lemma, which (combined with Lemma \ref{lemma:decomp} and the fact that $d(t)=\gcd\big(f(t),g(t)\big)$ is of degree 0) proves our theorem.

\begin{lemma}
\label{lemma:finish}
Let $f(t), g(t)\in \eqp$, and let $d(t)\defeq\gcd\big(f(t),g(t)\big)$ be of degree 0, let $H\subseteq\eqp$ be a finite set, and let
\[S(t)\defeq\bigcup_{h\in H} \big(h(t)+\ideal{f(t),g(t)}\big).\]
The function $F(t)$ giving the largest integer not in $S(t)$ is in $\eqp$, with degree at most $\max\{\deg(f)+\deg(g),\deg(h)\}$.
\end{lemma}

In our case, this gives that $F$ is of degree at most 2.

\section{Proofs of Lemmas}
\begin{proof}[Proof of Lemma \ref{lemma:canon}]
Part 1 is a standard result, using the extended Euclidean algorithm and the fact that if $c=p'a+q'b$ is one solution, then all solutions are given by $c=(p'-kb)a+(q'+ka)b$ for $k\in\Z$.

For part 2, we have $c=pa+qb$ with $0\le p <b$. The reverse implication is immediate: if $q\ge 0$, then we have written $c=pa+qb$ with $p,q\in \Z_{\ge 0}$, proving that $c\in\ideal{a,b}$. Conversely, suppose $c\in \ideal{a,b}$ so that $c=p'a+q'b$ with $p',q'\in\Z_{\ge 0}$. By part 1, if $k$ and $r$ are the quotient and remainder when $p'$ is divided by $b$, then
\[c=ra+(q'+ka)b\]
is the canonical form for c, with $q=q'+ka\ge 0$.
\end{proof}

\begin{proof}[Proof of Lemma \ref{lemma:tools}]
As these are proved in Section 4 of \cite{CLS12}, we simply give an outline here.

For part 1, this is mostly obvious: simply take component polynomials of the quasi-polynomial. The main subtlety, as seen in the example from Section 2, is that integer-valued polynomials may have non-integral coefficients. In this case, let $m$ be the least common multiple of the denominators of the coefficients. Examining $f(ms+i)$, we see that all coefficients of $s^k$ must be integral, except possibly the constant coefficient; since the function is integer-valued, the constant coefficient must also be integral.

For part 2, simply perform polynomial division. The main subtlety is the following: Suppose, for example, that $f(t)\defeq t^2+3t$ and $g(t)\defeq 2t+1$. Then the leading coefficient of $g$ does not divide the leading coefficient of $f$, and the traditional polynomial division algorithm would produce quotients that are not integer-valued. Instead, we look separately at each residue class of $t$ modulo  the leading coefficient of $g$; for example, if $t$ is odd, then $t=2s+1$ for some $s\in\Z_{\ge 0}$, so substituting gives $f(2s+1)=4s^2+10s+3$ and $g(2s+1)=4s+3$, and now the leading term does divide evenly.

For part 3, first perform the polynomial division algorithm for part 2. For example, suppose $f(t)\defeq 2t-3$ and $g(t)\defeq t$, yielding $f=2g+ -3$. For part 3, however, we want the integer division algorithm: $f(t)=1g(t)+(t-3)$, and the remainder $t-3$ is between 0 and $g$ as long as $t\ge 3$. In other words, if we have found $f=q'g+r'$ with $\deg(r')<\deg(g)$, but we eventually have $r'(t)<0$, then we should use quotient $q\defeq q'-1$ and remainder $r\defeq g+r'$ instead, as eventually $0\le g(t)+r'(t)<g(t)$.

For part 4, run the extended Euclidean algorithm, repeatedly performing the division algorithm of part 2.

For part 5, note that given two polynomials, one eventually dominates the other.
\end{proof}

\begin{proof}[Proof of Lemma \ref{lemma:degen}]
For each $i$, let $\gamma_i\defeq\gcd(\alpha_i, \beta_i)\in\Z_+$, let $\alpha'_i\defeq\alpha_i/\gamma_i$, and let $\beta'_i\defeq\beta_i/\gamma_i$, with $\alpha'_i$ and $\beta'_i$ relatively prime. Dividing $\beta_i\alpha_j=\beta_i\alpha_i$ by $\gamma_i\gamma_j$ yields $\beta'_i\alpha'_j=\beta'_j\alpha'_i$. Since $\alpha'_j$ divides $\beta'_j\alpha'_i$ and is relatively prime to $\beta'_j$, we have that $\alpha'_j$ divides $\alpha'_i$. Similarly, $\alpha'_i$ divides $\alpha'_j$, and so $\alpha'_i=\alpha'_j$. Similarly, $\beta'_i=\beta'_j$. Taking $\alpha_0$ to be the common $\alpha'_i$ (equal for all $i$) and $\beta_0$ to be the common $\beta'_i$, the proof follows.
\end{proof}

\begin{proof}[Proof of Lemma \ref{lemma:decomp}]

Let $r\defeq\beta_n\alpha_1-\beta_1\alpha_n$. We are given that $r>0$ and that $\beta_j\alpha_i-\beta_i\alpha_j\ge 0$ for all $i\le j$. Let
\[H=\setBuilder{\lambda_2a_2+\cdots+\lambda_{n-1}a_{n-1}}{\lambda_i\in\Z_{\ge 0},\,\lambda_i<r}.\]
We will prove that $H$ has the required properties; certainly every element of $H$ is of degree at most 1. We must show that, if $u\in \ideal{a_1(t),\ldots,a_n(t)}$, then there
exists a $\lambda\in\Z_{\ge 0}^n$ such that $u=\lambda_1a_1(t)+\cdots+\lambda_na_n(t)$ and $\lambda_i<r$ for
$2\le i\le n-1$.

By definition of the semigroup, there exists a $\lambda\in\Z_{\ge 0}^n$ such that
$u=\lambda_1a_1(t)+\cdots+\lambda_na_n(t)$. Suppose that $\lambda_i\geq r$ for some $i$ with
$2\le i\le n-1$.  Let
$p=\beta_n\alpha_i-\beta_i\alpha_n$ and $q=\beta_i\alpha_1-\beta_1\alpha_i$ and observe that $p,q\ge0$.
Then
\begin{align*}
pa_1(t)+qa_n(t) &=
(\beta_n\alpha_i-\beta_i\alpha_n)(\alpha_1 t + \beta_1) +
  (\beta_i\alpha_1-\beta_1\alpha_i)(\alpha_n t + \beta_n) \\ &=
\beta_n\alpha_i\alpha_1t - \beta_i\alpha_n\beta_1 +
  \beta_i\alpha_1\beta_n - \beta_1\alpha_i\alpha_nt \\&=
(\beta_n\alpha_1-\beta_1\alpha_n)(\alpha_it+\beta_i)\\
&= ra_i(t).
\end{align*}
Now define $\lambda'$ by $\lambda'_1\defeq\lambda_1+p$, $\lambda'_i\defeq\lambda_i-r$, $\lambda'_n\defeq\lambda_n+q$, and $\lambda'_j\defeq\lambda_j$ for all other $j$. Then $$u=\lambda'_1a_1(t)+\cdots+\lambda'_na_n(t)$$
is a new representation of $u$. Repeating this
process eventually yields a representation with $\lambda_i<r$ for $2\le i\le n-1$.
\end{proof}

\begin{proof}[Proof of Lemma \ref{lemma:oneh}]
Assume $0\le p< g(t)$ and $q<0$. We need conditions on $p,q$ for which $pf(t)+qg(t)-h(t)\in\ideal{f(t),g(t)}$. By Lemma \ref{lemma:tools}(4) and the fact that $f(t)$ and $g(t)$ are relatively prime, we may write $1=r'(t)f(t)+s'(t)g(t)$, where $r',s'\in\eqp$. Multiplying this equation by $h(t)$ yields 
\[h(t)=\big(h(t)r'(t)\big)f(t)+\big(h(t)s'(t)\big)g(t).\]
If we let $k(t)$ and $r(t)$ be the quotient and remainder when $h(t)r'(t)$ is divided by $g(t)$, using Lemma \ref{lemma:tools}(3), and let $s(t)=h(t)s'(t)+k(t)f(t)$, then we have
\[h(t)=r(t)f(t)+s(t)g(t),\]
with $0\le r(t)<g(t)$. In particular, this gives us that $\deg(r)\le \deg(g)$ and therefore $\deg(s)\le\max\{\deg(f),\deg(h)-\deg(g)\}$.
We have
\begin{equation}\label{eqn:canon}pf(t)+qg(t)-h(t)=\big(p-r(t)\big)f(t)+\big(q-s(t)\big)g(t).\end{equation}

\proofCase{p\ge r(t)} Then $0\le p-r(t)< g(t)$, and (\ref{eqn:canon}) is in canonical form. Then $pf(t)+qg(t)-h(t)\in\ideal{f(t),g(t)}$ if and only if $q-s(t)\ge 0$, by Lemma \ref{lemma:canon}(2).

\proofCase{p<r(t)} Then the canonical form for $pf(t)+qg(t)-h(t)$ will be 
\begin{equation}\label{eqn:canon2} \big(p-r(t)+g(t)\big)f(t)+\big(q-s(t)-f(t)\big)g(t), \end{equation}
since $0\leq p<r(t)$ and $r(t)<g(t)$ imply that
\[ 0 < p-r(t)+g(t) < g(t). \]
Since $h(t)$, $r(t)$, $f(t)$, and $g(t)$ are eventually positive and $r(t)<g(t)$, we eventually have
\[0<h(t)=r(t)f(t)+s(t)g(t)<g(t)\big(f(t)+s(t)\big),\]
and so eventually $f(t)+s(t)>0$. Since we are assuming that $q<0$, eventually $q-s(t)-f(t)<0$. Therefore the canonical form  (\ref{eqn:canon2}) shows that $pf(t)+qg(t)-h(t)\notin \ideal{f(t),g(t)}$, by Lemma \ref{lemma:canon}(2).

Combining the two cases, we see that $pf(t)+qg(t)\in h(t)+\ideal{f(t),g(t)}$ exactly when $p\ge r(t)$ and $q\ge s(t)$, as desired.
\end{proof}

\begin{proof}[Proof of Lemma \ref{lemma:finish}]
Since $d(t)\in\eqp$ is of degree zero, it is eventually a periodic function. By Remark \ref{remark:wlog}, we may focus on a component of $d(t)$ and assume that $d(t)=d$ is a constant.

\proofCase{d=1}
We assume, without loss of generality, that $0\in H$ and the other functions in $H$ are eventually positive: if not, let $h_0(t)$ be the eventually minimal polynomial in $H$, and find the largest integer $F'(t)$ not in
\[S'(t)=\bigcup_{h\in H}\bigg(\big(h(t)-h_0(t)\big)+\ideal{f(t),g(t)}\bigg);\]
then $F(t)=F'(t)+h_0(t)$.

We wish to describe the set $U(t)$ of $(p,q)$ corresponding to canonical forms $u=pf(t)+qg(t)$ such that $u\notin S$. Canonical implies that $0\le p<g(t)$, and $u\notin \left(0+\ideal{f(t),g(t)}\right)$ implies that $q<0$, by Lemma \ref{lemma:canon}(2). Each $h\in H\setminus\{0\}$ gives the condition ``$p<r_h(t)$ or $q<s_h(t)$'', by Lemma \ref{lemma:oneh}. Then the set $U(t)$ has a ``stairstep'' shape as in Figure \ref{fig:ex}. To be precise, order the $\big(r_h(t),s_h(t)\big)$ such that $r_1(t)\le r_2(t)\le\cdots\le r_m(t)$ (eventually). Notice that if $r_i(t)\le r_j(t)$, then we may assume without loss of generality that $s_i(t)>s_j(t)$ (eventually), or else the condition ``$p<r_j(t)$ or $q<s_j(t)$'' would be redundant. Further include $\big(r_0(t),s_0(t)\big)=(0,0)$ (corresponding to $u\notin 0+\ideal{f(t),g(t)}$) and $\big(r_{m+1}(t),s_{m+1}(t)\big)= \big(g(t),-f(t)\big)$ (corresponding to $(p,q)$ being canonical). Finally, for $0\le i\le m$, let $\big(\alpha_i(t),\beta_i(t)\big)=\big(r_{i+1}(t)-1,s_i(t)-1\big)$. Then the set of $(p,q)\in U(t)$ is exactly the set such that $0\le p$ and
\[(p,q)\le \big(\alpha_0(t),\beta_0(t)\big)\quad\text{or}\quad\cdots\quad\text{or}\quad(p,q)\le \big(\alpha_m(t),\beta_m(t)\big).\]
Since if $(p,q)\le\big(\alpha_i(t),\beta_i(t)\big)$, then $pf(t)+qg(t)\le \alpha_i(t)f(t)+\beta_i(t)g(t)$, our candidates for the largest integer not in $S(t)$ are the $\alpha_i(t)f(t)+\beta_i(t)g(t)$, $0\le i\le m$. Then our final answer is
\[\max_{0\le i\le m} \alpha_i(t)f(t)+\beta_i(t)g(t),\]
which is in $\eqp$, by Lemma \ref{lemma:tools}(5). Furthermore, using that $\deg(\alpha_i)=\deg(r_i)\le \deg(g)$ and $\deg(\beta_i)=\deg(s_i)\le\max\{\deg(f),\deg(h)-\deg(g)\}$, by Lemma \ref{lemma:oneh}, we have that our final degree is at most
$\max\{\deg(f)+\deg(g),\deg(h)\}$.

\proofCase{d> 1}
Let $H_j(t)\defeq\setBuilder{h\in H}{h(t)\equiv
  j\bmod d}$. The remainder when a given $h(t)$ is divided by $d$ is a periodic function, by Lemma \ref{lemma:tools}(3); applying Remark \ref{remark:wlog} as necessary, we may assume that these remainders are constant, that is, that $H_j=H_j(t)$  does not depend on $t$. Let $S_j(t)=\setBuilder{u\in
  S(t)}{u\equiv j\bmod d}$. Since $d$ divides every element of $\ideal{f(t),g(t)}$, 
  \[
S_j(t) = \bigcup_{h\in H_j}\bigg(h(t) + \ideal{f(t),g(t)}\bigg)
.\]
Let $F_j(t)$ be the largest integer not in
\[
\frac{1}{d}(S_j(t)-j) =
\bigcup_{h\in H_j}
\left(\frac{h(t)-j}{d}+\left\langle\frac{f(t)}{d},\frac{g(t)}{d}\right\rangle\right)
.\]
By the case $d=1$ proved above, $F_j(t)$ is in $\eqp$. Since $dF_j(t)+j$ is the maximum $u\in j+d\Z$ such
that $u\notin S_j(t)$, we get that
$F(t)=\max_{j}\left(dF_j(t)+j\right)$. Lemma \ref{lemma:tools}(5) then implies that
$F(t)$ is in $\eqp$.
\end{proof}

\section{Proof of Corollary \ref{cor:wagon}}
Without loss of generality, we may assume that $b_1=0$ (substituting $s=t-b_1$, if necessary). For a fixed $j\in\Z$, look at all $t\equiv j\bmod b_n$, and let $s$ be such that  $t=b_ns+j$. Let $a_i(s)\defeq t+b_i=b_ns+j+b_i$. We must prove that $F\big(a_1(s),\ldots,a_n(s)\big)$ is a \emph{polynomial} in $s$. We follow the proof of Theorem \ref{thm:main} and of the lemmas, looking for places where periodicity might be introduced.  Note that the $a_i(s)$ are correctly ordered (as defined before the statement of Lemma \ref{lemma:degen}) so that Lemma \ref{lemma:decomp} gives us
\[S=\bigcup_{h\in H}\big(h(s)+\ideal{a_1(s),a_n(s)}\big).\]

Following the proof of Lemma \ref{lemma:decomp}, note that each $h\in H$ is a nonnegative integer combination of $a_2(s),\ldots,a_{n-1}(s)$. In particular, each $h\in H$ has the form $kb_ns+\ell$ for some integers $k$ and $\ell$.

Focus on a specific $h=kb_ns+\ell$. Let
\[d(s)\defeq\gcd\big(a_1(s),a_n(s)\big)=\gcd(b_ns+j,b_ns+j+b_n)=\gcd(b_n,b_ns+j)=\gcd(j,b_n),\]
which is a constant. Finding $p$ and $q$ with $pj+qb_n=d$ gives us
\[d=(ps-q+p)a_1(s)-(ps-q)a_n(s).\]
Assume for the moment that $d=1$. We next examine the proof of Lemma \ref{lemma:oneh}. The first step is to take 
\begin{align*}
h&=h\cdot 1=h\cdot\big((ps-q+p)a_1(s)-(ps-q)a_n(s)\big)\\
&=\big((kb_ns+\ell)(ps-q+p)\big)a_1(s) - \big((kb_ns+\ell)(ps-q+p)\big)a_n(s),
\end{align*}
and write it in canonical form. This involves taking the remainder when dividing
\[(kb_ns+\ell)(ps-q+p)=kpb_ns^2+(\ell p-kb_nq+kb_np)s+(-\ell q+\ell p)\]
by $a_n(s)=b_ns+j+b_n$. The first step of polynomial long division works over the integers, because the leading coefficient, $b_n$, of $a_n(s)$ divides into the leading coefficient, $kpb_n$, of the dividend. This leaves a remainder that is linear in $s$. One more step of long division (dividing a linear in $s$ function by a linear in $s$ function and taking the linear in $s$ remainder) will then lead to the final answer, without introducing extra periodicity.

No other steps in the entire proof have a possibility of adding periodicity, so we are done, in the case $d=1$.

If $d>1$, following the proof of Lemma \ref{lemma:finish} in the $d>1$ case, let $b'_n=b_n/d$, $j'=j/d$, and $\ell'=\floor{\ell/d}$ be integers. Then the reduction to the $d=1$ case gives that we must examine,
\[\floor{\frac{h(s)}{d}}+\ideal{\frac{a_1(s)}{d},\frac{a_n(s)}{d}}=(kb'_ns+\ell')+\ideal{b'_ns+j',b'_ns+j'+b'_n},\]
and we have reduced to the $d=1$ case.

\section{Proof of Theorem \ref{thm:3gens}}
The $n=1$ case is trivial: $F\big(a_1(t)\big)=-a_1(t)$ is the usual interpretation.

The $n=2$ case follows from Sylvester's formula:
\[F\big(a_1(t),a_2(t)\big)=\lcm\big(a_1(t),a_2(t)\big)-a_1(t)-a_2(t)=\frac{a_1(t)a_2(t)}{\gcd\big(a_1(t),a_2(t)\big)}-a_1(t)-a_2(t),\]
and the gcd can be computed using Lemma \ref{lemma:tools}.

For the $n=3$ case, we show that the steps of R\o dseth's algorithm \cite{Rodseth78} can be performed on eventual quasi-polynomials. We enumerate the steps of the algorithm as described for integers, and follow each step with a comment on how it works with elements of $\eqp$. 

\begin{enumerate}
\item[(1)] If $d(t)\defeq\gcd\big(a_1(t),a_2(t),a_3(t)\big)\ne 1$, then use that
\[F\big(a_1(t),a_2(t),a_3(t)\big)=d(t)F\left(\frac{a_1(t)}{d(t)},\frac{a_2(t)}{d(t)},\frac{a_3(t)}{d(t)}\right).\]
\end{enumerate}

For elements of $\eqp$, Lemma \ref{lemma:tools}(4) allows us to compute this gcd.

\begin{enumerate}
\item[(2)] We may assume that $a_1(t), a_2(t),a_3(t)$ are relatively prime where $a_1(t),a_2(t),a_3(t)\in\eqp$. If $e(t)\defeq \gcd\big(a_1(t),a_2(t)\big)\ne 1$, then use that
\[F\big(a_1(t),a_2(t),a_3(t)\big)=e(t)F\left(\frac{a_1(t)}{e(t)},\frac{a_2(t)}{e(t)},a_3(t)\right)+a_3(t)\big(e(t)-1\big).\]
\end{enumerate}
\noindent This can again be done using Lemma \ref{lemma:tools}(4).

\begin{enumerate}
\item[(3)] We may assume that $a_1(t)$ and $a_2(t)$ are relatively prime.
Compute $s_0(t)\in\eqp$ such that $a_2(t)s_0(t)\equiv a_3(t)\bmod{a_1(t)}$ and $0\le s_0(t)< a_1(t)$.
\end{enumerate}
\noindent To do this for elements of $\eqp$, first use Lemma \ref{lemma:tools}(4) to find $p,q\in\eqp$ such that $1=p(t)a_1(t)+q(t)a_2(t)$. Then
\[a_2(t)\big(a_3(t)q(t)\big)=a_3(t)\big(q(t)a_2(t)\big)\equiv a_3(t)\cdot 1\bmod{a_1(t)}.\]
Now let $s_0(t)$ be the remainder when $a_3(t)q(t)$ is divided by $a_1(t)$, using Lemma \ref{lemma:tools}(3). This is the desired $s_0(t)$.

\begin{enumerate}
\item[(4)] If $s_0(t)=0$, then $a_3(t)$ is a multiple of $a_1(t)$, and the Frobenius problem reduces to the $n=2$ case.
\end{enumerate}
\noindent If some of the components of $s_0(t)\in\eqp$ are zero, we will have to split into cases, one for each component.

\begin{enumerate}
\item[(5)] Compute $q_1,s_1\in\eqp$ such that
\[a_1(t)=q_1(t)s_0(t)-s_1(t),\quad 0\le s_1(t)<s_0(t).\]
\end{enumerate}
\noindent If $\deg(s_0)=0$, this is a slight variant of the usual integer division algorithm. If $\deg(s_0)>0$, use Lemma \ref{lemma:tools}(2) to compute $a_1(t)=q(t)s_0(t)+r(t)$, with $\deg(r) <\deg(s_0)$. If $r(t)$ is eventually negative or zero, simply set $q_1(t)=q(t)$ and $s_1(t)=-r(t)$, and we have that eventually $0\le s_1(t)<s_0(t)$. If $r(t)$ is eventually positive, set $q_1(t)=q(t)+1$ and $s_1(t)=s_0(t)-r(t)$, which will eventually satisfy the bounds.

\begin{enumerate}
\item[(6)] Continue computing
\begin{align*}
s_0(t)&=q_2(t)s_1(t)-s_2(t), \quad0\le s_2(t)<s_1(t),\\
s_1(t)&=q_3(t)s_2(t)-s_3(t), \quad 0\le s_3(t)<s_2(t),\\
\ &\ \, \vdots\\
s_{m-2}(t)&=q_{m}(t)s_{m-1}(t)-s_m(t), \quad 0\le s_m(t)<s_{m-1}(t),\\
s_{m-1}(t)&=q_{m+1}(t)s_m(t)+s_{m+1},\quad s_{m+1}=0.
\end{align*}
\end{enumerate}
\noindent This is simply repeating the process from Step 5. We must establish that it terminates. It suffices to show that, if $\deg(s_i)>0$, then $\deg(s_{i+2})<\deg(s_i)$ (and then once $\deg(s_i)=0$, the integers $s_i$ strictly decrease, so this will eventually terminate). Indeed, as described in Step 5, Lemma \ref{lemma:tools}(2) gives us $s_{i-1}(t)=q(t)s_i(t)+r(t)$, with $\deg(r) <\deg(s_i)$. If $r(t)$ is eventually negative or zero, we have $s_{i+1}(t)=-r(t)$, and the degree has decreased. If $r(t)$ is eventually positive, we have $s_{i+1}(t)=s_{i}(t)-r(t)$, which still have the same degree as $s_i$. But in the next step, $s_{i+1}(t)=1\cdot s_i(t)-r(t)$, and so $s_{i+2}(t)=r(t)$ has lower degree. Of course, different components of the eventual quasi-polynomials may terminate at different $m$, and we  must analyze each separately.

\begin{enumerate}
\item[(7)] Define $s_{-1}(t)=a_1(t)$, $P_{-1}(t)=0$, $P_0(t)=1$, and recursively $P_{i+1}(t)=q_{i+1}(t)P_i(t)-P_{i-1}(t)$. Then 
\[0=\frac{s_{m+1}(t)}{P_{m+1}(t)}<\frac{s_m(t)}{P_m(t)}<\cdots<\frac{s_0(t)}{P_0(t)}<\frac{s_{-1}(t)}{P_{-1}(t)}=\infty.\]
Determine the unique index $i$ such that
\[\frac{s_{i+1}(t)}{P_{i+1}(t)}\le \frac{a_3(t)}{a_2(t)}<\frac{s_i(t)}{P_i(t)}.\]
\end{enumerate}
\noindent Restricting to a fixed residue class modulo the period of the quasi-polynomials, such an index $i$ exists, since each $s_i(t)/P_i(t)$ is a rational function, so is eventually greater than, eventually less than, or eventually equal to $a_3(t)/a_2(t)$.

\begin{enumerate}
\item[(8)] Then
\begin{align*}
F\big(a_1(t),a_2(t),a_3(t)\big)=-a_1&(t)+a_2(t)\left(s_i(t)-1\right)+a_3(t)\big(P_{i+1}(t)-1\big)\\
& - \min\big\{a_2(t)s_{i+1}(t),a_3(t)P_i(t)\big\}.
\end{align*}
\end{enumerate}
\noindent This is in $\eqp$, by Lemma \ref{lemma:tools}(5).


\begin{thebibliography}{10}

\bibitem{CW11}
Danny Calegari and Alden Walker.
\newblock Integer hulls of linear polyhedra and scl in families.
\newblock {\em Transactions of the American Mathematical Society}, 365:5085--5102, 2011.

\bibitem{CLS12}
Sheng Chen, Nan Li, and Steven~V. Sam.
\newblock Generalized {E}hrhart polynomials.
\newblock {\em Trans. Amer. Math. Soc.}, 364(1):551--569, 2012.

\bibitem{Ehrhart62}
Eug{\`e}ne Ehrhart.
\newblock Sur les poly\`edres rationnels homoth\'etiques \`a {$n$}\ dimensions.
\newblock {\em C. R. Acad. Sci. Paris}, 254:616--618, 1962.

\bibitem{ELSW07}
David Einstein, Daniel Lichtblau, Adam Strzebonski, and Stan Wagon.
\newblock Frobenius numbers by lattice point enumeration.
\newblock {\em Integers}, 7:A15, 63, 2007.

\bibitem{RA05}
J.~L. Ram{\'{\i}}rez~Alfons{\'{\i}}n.
\newblock {\em The {D}iophantine {F}robenius problem}, volume~30 of {\em Oxford
  Lecture Series in Mathematics and its Applications}.
\newblock Oxford University Press, Oxford, 2005.

\bibitem{Roberts56}
J.~B. Roberts.
\newblock Note on linear forms.
\newblock {\em Proc. Amer. Math. Soc.}, 7:465--469, 1956.

\bibitem{Rodseth78}
{\O}ystein~J. R{\o}dseth.
\newblock On a linear {D}iophantine problem of {F}robenius.
\newblock {\em J. Reine Angew. Math.}, 301:171--178, 1978.

\bibitem{Sylvester84}
James~J. Sylvester.
\newblock Mathematical questions with their solutions.
\newblock {\em Educational Times}, 41(21), 1884.

\bibitem{W15}
Stan Wagon.
\newblock personal comunication.

\bibitem{Woods14}
Kevin Woods.
\newblock The unreasonable ubiquitousness of quasi-polynomials.
\newblock {\em Electron. J. Combin.}, 21(1):Paper 1.44, 23, 2014.

\bibitem{Woods15}
Kevin Woods.
\newblock {P}resburger arithmetic, rational generating functions, and
  quasi-polynomials.
\newblock \emph{Journal of
  Symbolic Logic}, 80:433--449, 2015. Extended abstract in \emph{ICALP 2013}.  

\end{thebibliography}

\end{document}